\documentclass[12pt,oneside]{amsart}  
\usepackage{amssymb}
\usepackage{amsfonts}
\usepackage{amsthm}
\usepackage{amsthm}
\usepackage{amscd}
\usepackage{array}

\reversemarginpar         
\numberwithin{equation}{section}
\theoremstyle{plain}
\newtheorem{theorem}{Theorem}[section]

\newtheorem{lemma}[theorem]{Lemma}
\theoremstyle{definition}

\theoremstyle{remark}

\begin{document}
%

\newcommand{\MgNekp}{\mathcal{M}_{g,N+1}^{(k,p)}} 
\newcommand{\M}{\mathcal{M}_{g,N+1}^{(1)}}
\newcommand{\Teich}{\mathcal{T}_{g,N+1}^{(1)}}
\newcommand{\T}{\mathrm{T}}
\newcommand{\corr}{\bf}
\newcommand{\vac}{|0\rangle}
\newcommand{\Ga}{\Gamma}
\newcommand{\new}{\bf}
\newcommand{\define}{\def}
\newcommand{\redefine}{\def}
\newcommand{\Cal}[1]{\mathcal{#1}}
\renewcommand{\frak}[1]{\mathfrak{{#1}}}
\newcommand{\refE}[1]{(\ref{E:#1})}
\newcommand{\refS}[1]{Section~\ref{S:#1}}
\newcommand{\refSS}[1]{Section~\ref{SS:#1}}
\newcommand{\refT}[1]{Theorem~\ref{T:#1}}
\newcommand{\refO}[1]{Observation~\ref{O:#1}}
\newcommand{\refP}[1]{Proposition~\ref{P:#1}}
\newcommand{\refD}[1]{Definition~\ref{D:#1}}
\newcommand{\refC}[1]{Corollary~\ref{C:#1}}
\newcommand{\refL}[1]{Lemma~\ref{L:#1}}
\newcommand{\R}{\ensuremath{\mathbb{R}}}
\newcommand{\C}{\ensuremath{\mathbb{C}}}
\newcommand{\N}{\ensuremath{\mathbb{N}}}
\newcommand{\Q}{\ensuremath{\mathbb{Q}}}
\renewcommand{\P}{\ensuremath{\mathcal{P}}}
\newcommand{\Z}{\ensuremath{\mathbb{Z}}}
\newcommand{\kv}{{k^{\vee}}}
\renewcommand{\l}{\lambda}
\newcommand{\gb}{\overline{\mathfrak{g}}}
\newcommand{\hb}{\overline{\mathfrak{h}}}
\newcommand{\g}{\mathfrak{g}}
\newcommand{\h}{\mathfrak{h}}
\newcommand{\gh}{\widehat{\mathfrak{g}}}
\newcommand{\ghN}{\widehat{\mathfrak{g}_{(N)}}}
\newcommand{\gbN}{\overline{\mathfrak{g}_{(N)}}}
\newcommand{\tr}{\mathrm{tr}}
\newcommand{\sln}{\mathfrak{sl}}
\newcommand{\sn}{\mathfrak{s}}
\newcommand{\so}{\mathfrak{so}}
\newcommand{\spn}{\mathfrak{sp}}
\newcommand{\tsp}{\mathfrak{tsp}(2n)}
\newcommand{\gl}{\mathfrak{gl}}
\newcommand{\slnb}{{\overline{\mathfrak{sl}}}}
\newcommand{\snb}{{\overline{\mathfrak{s}}}}
\newcommand{\sob}{{\overline{\mathfrak{so}}}}
\newcommand{\spnb}{{\overline{\mathfrak{sp}}}}
\newcommand{\glb}{{\overline{\mathfrak{gl}}}}
\newcommand{\Hwft}{\mathcal{H}_{F,\tau}}
\newcommand{\Hwftm}{\mathcal{H}_{F,\tau}^{(m)}}

\newcommand{\car}{{\mathfrak{h}}}    
\newcommand{\bor}{{\mathfrak{b}}}    
\newcommand{\nil}{{\mathfrak{n}}}    
\newcommand{\vp}{{\varphi}}
\newcommand{\bh}{\widehat{\mathfrak{b}}}  
\newcommand{\bb}{\overline{\mathfrak{b}}}  
\newcommand{\Vh}{\widehat{\mathcal V}}
\newcommand{\KZ}{Kniz\-hnik-Zamo\-lod\-chi\-kov}
\newcommand{\TUY}{Tsuchia, Ueno  and Yamada}
\newcommand{\KN} {Kri\-che\-ver-Novi\-kov}
\newcommand{\pN}{\ensuremath{(P_1,P_2,\ldots,P_N)}}
\newcommand{\xN}{\ensuremath{(\xi_1,\xi_2,\ldots,\xi_N)}}
\newcommand{\lN}{\ensuremath{(\lambda_1,\lambda_2,\ldots,\lambda_N)}}
\newcommand{\iN}{\ensuremath{1,\ldots, N}}
\newcommand{\iNf}{\ensuremath{1,\ldots, N,\infty}}

\newcommand{\tb}{\tilde \beta}
\newcommand{\tk}{\tilde \kappa}
\newcommand{\ka}{\kappa}
\renewcommand{\k}{\varkappa}

\newcommand{\Pif} {P_{\infty}}
\newcommand{\Pinf} {P_{\infty}}
\newcommand{\PN}{\ensuremath{\{P_1,P_2,\ldots,P_N\}}}
\newcommand{\PNi}{\ensuremath{\{P_1,P_2,\ldots,P_N,P_\infty\}}}
\newcommand{\Fln}[1][n]{F_{#1}^\lambda}
\newcommand{\tang}{\mathrm{T}}
\newcommand{\Kl}[1][\lambda]{\can^{#1}}
\newcommand{\A}{\mathcal{A}}
\newcommand{\U}{\mathcal{U}}
\newcommand{\V}{\mathcal{V}}
\renewcommand{\O}{\mathcal{O}}
\newcommand{\Ae}{\widehat{\mathcal{A}}}
\newcommand{\Ah}{\widehat{\mathcal{A}}}
\newcommand{\La}{\mathcal{L}}
\newcommand{\Le}{\widehat{\mathcal{L}}}
\newcommand{\Lh}{\widehat{\mathcal{L}}}
\newcommand{\eh}{\widehat{e}}
\newcommand{\Da}{\mathcal{D}}
\newcommand{\kndual}[2]{\langle #1,#2\rangle}
\newcommand{\cins}{\frac 1{2\pi\mathrm{i}}\int_{C_S}}
\newcommand{\cinsl}{\frac 1{24\pi\mathrm{i}}\int_{C_S}}
\newcommand{\cinc}[1]{\frac 1{2\pi\mathrm{i}}\int_{#1}}
\newcommand{\cintl}[1]{\frac 1{24\pi\mathrm{i}}\int_{#1 }}
\newcommand{\w}{\omega}
\newcommand{\ord}{\operatorname{ord}}
\newcommand{\res}{\operatorname{res}}
\newcommand{\nord}[1]{:\mkern-5mu{#1}\mkern-5mu:}
\newcommand{\Fn}[1][\lambda]{\mathcal{F}^{#1}}
\newcommand{\Fl}[1][\lambda]{\mathcal{F}^{#1}}
\renewcommand{\Re}{\mathrm{Re}}

\newcommand{\ha}{H^\alpha}

\define\ldot{\hskip 1pt.\hskip 1pt}
\define\ifft{\qquad\text{if and only if}\qquad}
\define\a{\alpha}
\redefine\d{\delta}
\define\w{\omega}
\define\ep{\epsilon}
\redefine\b{\beta} \redefine\t{\tau} \redefine\i{{\,\mathrm{i}}\,}
\define\ga{\gamma}
\define\cint #1{\frac 1{2\pi\i}\int_{C_{#1}}}
\define\cintta{\frac 1{2\pi\i}\int_{C_{\tau}}}
\define\cintt{\frac 1{2\pi\i}\oint_{C}}
\define\cinttp{\frac 1{2\pi\i}\int_{C_{\tau'}}}
\define\cinto{\frac 1{2\pi\i}\int_{C_{0}}}
\define\cinttt{\frac 1{24\pi\i}\int_C}
\define\cintd{\frac 1{(2\pi \i)^2}\iint\limits_{C_{\tau}\,C_{\tau'}}}
\define\cintdr{\frac 1{(2\pi \i)^3}\int_{C_{\tau}}\int_{C_{\tau'}}
\int_{C_{\tau''}}}
\define\im{\operatorname{Im}}
\define\re{\operatorname{Re}}
\define\res{\operatorname{res}}
\redefine\deg{\operatornamewithlimits{deg}}
\define\ord{\operatorname{ord}}
\define\rank{\operatorname{rank}}
\define\fpz{\frac {d }{dz}}
\define\dzl{\,{dz}^\l}
\define\pfz#1{\frac {d#1}{dz}}

\define\K{\Cal K}
\define\U{\Cal U}
\redefine\O{\Cal O}
\define\He{\text{\rm H}^1}
\redefine\H{{\mathrm{H}}}
\define\Ho{\text{\rm H}^0}
\define\A{\Cal A}
\define\Do{\Cal D^{1}}
\define\Dh{\widehat{\mathcal{D}}^{1}}
\redefine\L{\Cal L}
\newcommand{\ND}{\ensuremath{\mathcal{N}^D}}
\redefine\D{\Cal D^{1}}
\define\KN {Kri\-che\-ver-Novi\-kov}
\define\Pif {{P_{\infty}}}
\define\Uif {{U_{\infty}}}
\define\Uifs {{U_{\infty}^*}}
\define\KM {Kac-Moody}
\define\Fln{\Cal F^\lambda_n}
\define\gb{\overline{\mathfrak{ g}}}
\define\G{\overline{\mathfrak{ g}}}
\define\Gb{\overline{\mathfrak{ g}}}
\redefine\g{\mathfrak{ g}}
\define\Gh{\widehat{\mathfrak{ g}}}
\define\gh{\widehat{\mathfrak{ g}}}
\define\Ah{\widehat{\Cal A}}
\define\Lh{\widehat{\Cal L}}
\define\Ugh{\Cal U(\Gh)}
\define\Xh{\hat X}
\define\Tld{...}
\define\iN{i=1,\ldots,N}
\define\iNi{i=1,\ldots,N,\infty}
\define\pN{p=1,\ldots,N}
\define\pNi{p=1,\ldots,N,\infty}
\define\de{\delta}

\define\kndual#1#2{\langle #1,#2\rangle}
\define \nord #1{:\mkern-5mu{#1}\mkern-5mu:}
\define \sinf{{\widehat{\sigma}}_\infty}
\define\Wt{\widetilde{W}}
\define\St{\widetilde{S}}
\newcommand{\SigmaT}{\widetilde{\Sigma}}
\newcommand{\hT}{\widetilde{\frak h}}
\define\Wn{W^{(1)}}
\define\Wtn{\widetilde{W}^{(1)}}
\define\btn{\tilde b^{(1)}}
\define\bt{\tilde b}
\define\bn{b^{(1)}}
%
\define\eps{\varepsilon}    
\define\doint{({\frac 1{2\pi\i}})^2\oint\limits _{C_0}
       \oint\limits _{C_0}}                            
\define\noint{ {\frac 1{2\pi\i}} \oint}   
\define \fh{{\frak h}}     
\define \fg{{\frak g}}     
\define \GKN{{\Cal G}}   
\define \gaff{{\hat\frak g}}   
\define\V{\Cal V}
\define \ms{{\Cal M}_{g,N}} 
\define \mse{{\Cal M}_{g,N+1}} 
\define \tOmega{\Tilde\Omega}
\define \tw{\Tilde\omega}
\define \hw{\hat\omega}
\define \s{\sigma}
\define \car{{\frak h}}    
\define \bor{{\frak b}}    
\define \nil{{\frak n}}    
\define \vp{{\varphi}}
\define\bh{\widehat{\frak b}}  
\define\bb{\overline{\frak b}}  
\define\Vh{\widehat V}
\define\KZ{Knizhnik-Zamolodchikov}
\define\ai{{\alpha(i)}}
\define\ak{{\alpha(k)}}
\define\aj{{\alpha(j)}}
\newcommand{\laxgl}{\overline{\mathfrak{gl}}}
\newcommand{\laxsl}{\overline{\mathfrak{sl}}}
\newcommand{\laxso}{\overline{\mathfrak{so}}}
\newcommand{\laxsp}{\overline{\mathfrak{sp}}}
\newcommand{\laxs}{\overline{\mathfrak{s}}}
\newcommand{\laxg}{\overline{\frak g}}
\newcommand{\bgl}{\laxgl(n)}
\newcommand{\tX}{\widetilde{X}}
\newcommand{\tY}{\widetilde{Y}}
\newcommand{\tZ}{\widetilde{Z}}

\vspace*{-1cm}
%
%
%
\vspace*{2cm}

\title[Lax operator algebras of type $G_2$]
{Lax operator algebras of type $G_2$}
\author[O.K. Sheinman]{Oleg K. Sheinman}


\begin{abstract}
Lax operator algebras for the root system $G_2$, and arbitrary finite genus Riemann surfaces and Tyurin data on them are constructed.
\end{abstract} \subjclass{17B66,
17B67, 14H10, 14H15, 14H55,  30F30, 81R10, 81T40} \keywords{
Current algebra, Lax operator algebra, exceptional Lie algebra $G_2$}
\maketitle

\tableofcontents

\section{Introduction}\label{S:intro}
Lax operator algebras are introduced in \cite{KSlax} and later investigated in different aspects in \cite{SSlax,ShN70AMS,Sh_lopa,Sh0910_Hamilt,Lax_prequant,Sch_Laxmulti}. For the moment, a most complete presentation of the theory of Lax operator algebras is given in \cite{Sh_DGr}.

Lax operator algebras constitute a certain class of almost graded current algebras on Riemann surfaces with marked points. As such they generalize loop and affine Krichever--Novikov algebras.

From the physical point of view Lax operator algebras can be characterized as the algebras of the theory of integrable systems while Krichever--Novikov algebras are the algebras of gauge and conformal symmetries of the 2-dimensional conformal field theory. There exists quite interesting interaction between Lax operator-\ and Krichever--Novikov algebras in this respect \cite{Lax_prequant,Sh_DGr}.

So far Lax operator algebras have been constructed only for classical simple Lie algebras (over $\C$) as the range of values of currents. The question whether there exist Lax operator algebras for exceptional simple Lie algebras was open. In this paper we construct such algebras for the exceptional Lie algebra $G_2$.

In \refS{G2} we recall the definition of the exceptional simple Lie algebra $G_2$ and give some relation heavily used below.

In \refS{LaxG2} we give the definition of Lax operator algebras corresponding to $G_2$, formulate the result on their closeness with respect to the Lie bracket (\refT{Liealg}), and outline its proof. The proof itself is given in the Appendix \ref{S:close}.

In \refS{almgrad} we define an almost graded structure on Lax operator algebras corresponding to $G_2$, formulate and prove a main result on them --- the \refT{almgrad}.

In \refS{centext} for every Lax operator algebra corresponding to $G_2$ we construct a certain 2-cocycle, and the corresponding central extension, and show that it is basically unique (\refT{central}).

We conclude with certain remarks given in \refS{conclude}.

All results of sections \ref{S:LaxG2} -- \ref{S:centext} are known for Lax operator algebras over classical simple Lie algebras,  but not over exceptional Lie algebras. In the case of $\g=G_2$ the proofs of those results require much more formal calculations. This is the reason of shifting almost all proofs into Appendix. We would like stress that though a similarity between the results and their proofs for different simple Lie algebras is evident, no concept-based proof exists. It is even not clear why such proof should exist (see more remarks in \refS{conclude}). It is a challenge to clarify this question.

I am thankful to N.Vavilov and M.Schlichenmaier for their interest having given me an additional motivation for thinking of the Lax operator algebras in the exceptional cases.
\section{Lie algebra $G_2$ in 7-dimensional representation}\label{S:G2}

According to \cite{Zhelob} $\g=G_2$ can be represented as the the Lie algebra of $7\times 7$-matrices of the form
\begin{equation}\label{E:G2matr}
\tilde A=\begin{pmatrix}
       0 & -\sqrt{2}a_2^t & -\sqrt{2}a_1^t \\
       \sqrt{2}a_1 & A & [a_2] \\
       \sqrt{2}a_2 & [a_1] & -A^t \\
       \end{pmatrix}
\end{equation}
where $a_1,a_2\in\C^3$ are interpreted as vector-columns, $a_1^t,a_2^t$ are the corresponding vector-rows, $A$ is a traceless $3\times 3$\,-matrix. For any $x\in\C^3$, $x^t=(x_1,x_2,x_3)$, by $[x]$ we denote the corresponding skew-symmetric $3\times 3$\,-matrix:
\[
  [x]=\begin{pmatrix}
       0 & x_3 & -x_2 \\
       -x_3 & 0 & x_1\\
       x_2 & -x_1 & 0 \\
       \end{pmatrix}.
\]
For any two vectors $x,y\in\C^3$ by $x\times y$ we denote their vector product. Then the following relations take place:
\begin{equation}\label{E:rels}
 \begin{split}
  &1^\circ\ [x]y=x\times y;\\
  &2^\circ\ [x][y]=yx^t-(x^ty)E;\\
  &3^\circ\ -[Ax]=A^t[x]+[x]A, \forall A\in \rm{Matr}(3\times 3),\ \tr A=0;\\
  &4^\circ\ [x\times y]=[x][y]-[y][x]=yx^t-xy^t.
 \end{split}
\end{equation}
All of them are easily proven by a straightforward calculation.
These four relations are everything that is necessary to prove that the matrices of the form \refE{G2matr} constitute a Lie algebra. We do it in the Appendix \ref{S:G2pr} for completeness. The same relations are the base of all subsequent calculations in this paper.
\section{A Lax operator algebra of the type $G_2$}
\label{S:LaxG2}

Let $\Sigma$ be a Riemann surface
with marked points $P_1,\ldots,P_N$, $Q_1,\ldots,Q_M$, $\ga_1,\ldots,\ga_K$. Assume every $\ga$-point to be associated with a 7-dimensional vector $\a=(0,\a_1^t,\a_2^t)$ where $\a_1,\a_2\in\C^3$ are interpreted as vector-columns, the upper $t$ denotes transposition. It may be illustrated
by the following picture.
\begin{figure}[h]\label{fig}
\begin{picture}(350,180)
\unitlength=1pt
\thicklines
%
\qbezier(100,20)(-25,70)(100,120) 
\qbezier(225,20)(350,70)(225,120) 
\qbezier(100,120)(162,138)(225,120) 
\qbezier(100,20)(162,2)(225,20) 
%
%
\qbezier(65,75)(65,45)(125,40) 
\qbezier(70,65)(100,70)(115,42) 
%
%
\qbezier(180,40)(240,45)(250,75) 
\qbezier(190,42)(210,85)(246,66) 
%
\put(50,70){\circle*{4}}
\put(50,55){\Large $P_1$}
\put(56,78){\circle*{2}}
\put(65,85){\circle*{2}}
\put(80,95){\circle*{2}}
\put(85,95){$P_N$}
%
\put(275,70){\circle*{2}}
\put(265,76){$Q_M$}
\put(271,62){\circle*{2}}
\put(262,52){\circle*{2}}
\put(245,41){\circle*{4}}
\put(227,32){\Large $Q_1$}
%
%
\put(130,70){\circle*{2}}
\put(136,78){\circle*{2}}
\put(145,85){\circle*{2}}
\put(160,95){\circle*{4}}
\put(180,105){\circle*{2}}
\put(200,112){\circle*{2}}
%
%
\thinlines
\put(160,95){\line(0,1){30}}
\put(155,125){\line(1,0){10}}
\put(155,125){\line(0,1){40}}
\put(165,125){\line(0,1){40}}
\put(155,165){\line(1,0){10}}
\put(164,89){\Large $\ga$}
\put(170,160){\Large $\a$}
\put(134,64){\Large $\ga_1$}
\put(204,106){\Large $\ga_K$}
\put(130,70){\line(0,1){30}}
\put(125,100){\line(1,0){10}}
\put(125,100){\line(0,1){40}}
\put(135,100){\line(0,1){40}}
\put(125,140){\line(1,0){10}}

\end{picture}
\caption{Tyurin data}
\end{figure}
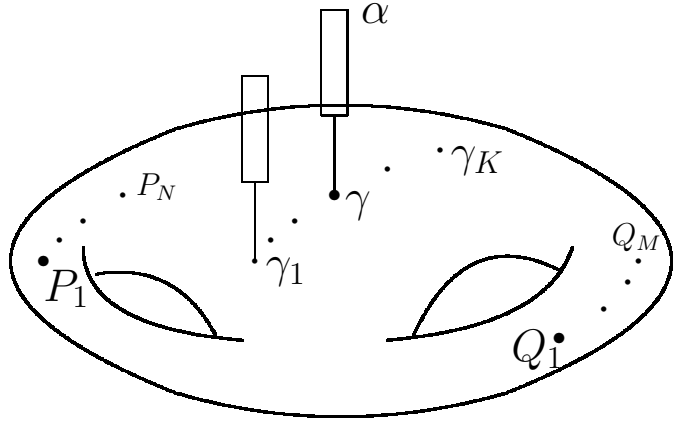
Following the lines of \cite{Klax,KSlax,Sh_DGr}, consider a $G_2$-valued
function $L$ on $\Sigma$, holomorphic outside $P_1,\ldots,P_N$, $Q_1,\ldots,Q_M$ and
$\ga_1,\ldots,\ga_K$, and having at most double poles at the
points in the last set. Assume that at every $\ga$-point $L$ possesses an expansion of the form
\begin{equation}\label{E:expan}
L(z)=\frac{L_{-2}}{z^2}+\frac{L_{-1}}{z}+L_0+L_1z+L_2z^2+\ldots\ .
\end{equation}
where $z$ is a local coordinate in the neighborhood of $\ga$ centered at $\ga$.

Let $\a_1,\a_2\in \C^3$ be fixed, $\b_1,\b_2\in \C^3$, and $\b_{01},\b_{02}\in\C$ be varying. Assume that the following orthogonality relations are fulfilled:
\begin{equation}\label{E:orth1}
\a_1^t\b_2=0, \quad \a_2^t\b_1=0, \quad \a_1^t\a_2=0.
\end{equation}

Let us take $L_0$ in the general form
\begin{equation}\label{E:L0}
L_0=\begin{pmatrix}
       0 & -\sqrt{2}a_2^t & -\sqrt{2}a_1^t \\
       \sqrt{2}a_1 & A & [a_2] \\
       \sqrt{2}a_2 & [a_1] & -A^t \\
       \end{pmatrix},
\end{equation}
and assume that
\begin{equation}\label{E:eigen}
\a_1^ta_2=0, \quad \a_2^ta_1=0, \quad A\a_1=\k_1\a_1,  \quad -A^t\a_2=\k_2\a_2.
\end{equation}
These conditions are referred to as eigenvalue conditions below, for the following reason. What is really assumed instead \refE{eigen} for classical Lie algebras, is the following eigenvalue relation: $L_0(0,\a_1^t,\a_2^t)^t=
{\tilde\k}(0,\a_1^t,\a_2^t)^t$. For $G_2$ it is equivalent to
\[
 a_1^t\a_2+a_2^t\a_1=0, \quad A\a_1+a_2\times\a_2={\tilde\k}\a_1, \quad -A^t\a_2+a_1\times\a_1={\tilde\k}\a_2.
\]
To obtain a right almost graded structure below we assume that, first, there are two eigenvalues ${\tilde\k}_1,{\tilde\k}_2$ instead one ${\tilde\k}$, and, second, that $a_1$ and $\a_2$, $a_2$ and $\a_1$ are orthogonal separately.
Then both $a_2$ and $\a_2$ are orthogonal to $\a_1$, hence $a_2\times\a_2=\l_1\a_1$, and similarly $a_1\times\a_1=\l_2\a_2$ where $\l_1,\l_2\in\C$. The last relations reduce to \refE{eigen} with $\k_1={\tilde\k}_1-\l_1$, $\k_2={\tilde\k}_2-\l_2$.

Take the residue of $L$ in the form
\begin{equation}\label{E:resid}
   L_{-1}=\begin{pmatrix}
       0 & -\sqrt{2}\b_{02}\a_2^t & -\sqrt{2}\b_{01}\a_1^t \\
       \sqrt{2}\b_{01}\a_1 & \a_1\b_2^t-\b_1\a_2^t & \b_{02}[\a_2] \\
       \sqrt{2}\b_{02}\a_2 & \b_{01}[\a_1] & \a_2\b_1^t-\b_2\a_1^t \\
       \end{pmatrix}
\end{equation}
Finally, take the $-2$ order term of the expansion in the form
\begin{equation}\label{E:L{-2}}
   L_{-2}=\mu\begin{pmatrix}
       0 & 0 & 0\\
       0 & \a_1\a_2^t &  0   \\
       0 &  0  & -\a_2\a_1^t    \\
       \end{pmatrix},\quad \mu\in\C .
\end{equation}
\begin{theorem}\label{T:Liealg}
The space of the $G_2$-valued meromorphic functions holomorphic outside $P_1,\ldots,P_N$, $Q_1,\ldots,Q_M$, $\ga_1,\ldots,\ga_K$, possessing expansions of the form
\refE{expan} at the $\ga$-points, and satisfying there the relations \refE{orth1}--\refE{L{-2}}, and the relation \refE{firstord} given below, constitute a Lie algebra with respect to the pointwise matrix commutator.
\end{theorem}
\begin{proof}
We give here only a sketch of the proof. All detailes are given in the Appendix \ref{S:close}.

We must prove that if $L$, $L'$ satisfy the conditions of the theorem then their  pointwise commutator $L^c=[L,L']$ does two. The main steps of the proof are as follows.

At the first step we prove that the order of $L^c$ at any $\ga$ is not less than $-2$, i.e. the $-3$, $-4$ order terms etc. are absent (\refSS{absence}).

Second, we prove that $L_{-1}^c$, $L_{-2}^c$ have the form \refE{resid}, \refE{L{-2}} respectively, and the relations \refE{orth1} keep true with the new values of $\b_1$, $\b_2$. To prove the latter we need one more assumption \refE{firstord}. See \refSS{L{-2}}, \refSS{resid} for details.

Third, we calculate $L_0^c$ and prove the eigenvalue condition \refE{eigen} for $L^c$ (\refSS{eigen}).

Finally, we prove that the condition \refE{firstord} holds true under commutation \refSS{firstord}.
\end{proof}

\section{Almost graded structure}
\label{S:almgrad}
The almost graded structure on associative and Lie algebras had been introduced by I.M.Krichever and S.P.Novikov  in \cite{KNFa}. For the Lax operator algebras and  the two point case it had been investigated in \cite{KSlax}. Most general setup of arbitrary numbers of incoming and outgoing points, for both Krichever--Novikov and Lax operator algebras has been considered by M.Schlichenmaier \cite{SLa,SLb,Sch_Laxmulti}. Here, we will follow this most general approach.

For every $m\in\Z$ consider a divisor
\begin{equation}\label{E:Dm}
  D_m=-m\sum\limits_{i=1}^NP_i+\sum_{j=1}^M(a_jm+b_{m,j})\,Q_i+2\sum_{s=1}^K\ga_s
\end{equation}
where $a_j,b_{m,j}\in\Q$, $a_j>0$, $a_jm+b_{m,j}$ is an ascending $\Z$\ -valued function of $m$, and there exists a $B\in\R_+$ such that
$|b_{m,j}|\le B, \forall m\in\Z, j=1,\ldots,M$\label{bjmbound}.
We require that
\begin{equation}\label{E:condd}
\sum_{j=1}^Ma_j=N,\qquad
\sum_{i=j}^Mb_{m,j}=N+g-1.
\end{equation}
Let
\begin{equation}\label{E:homos}
   \L_m=\{L\in\L\ |(L)+D_m\ge 0\},
\end{equation}
where $(L)$ is the divisor of a $\g$-valued function $L$.
To be more specific of $(L)$, let us notice that by order of a
meromorphic matrix-valued function we mean the minimal order of
its entries.

We call $\L_m$ {\it the (homogeneous) subspace of degree
$m$}\index{Homogeneous element (subspace)} of the Lie algebra
$\L$.
\begin{theorem}\label{T:almgrad}
\begin{itemize}
  \item[]
  \item[$1^\circ $] $ \quad \dim\,\L_m=(\dim\,\g)N$.
  \item[$2^\circ$] $\quad \L=\bigoplus\limits_{m=-\infty}^\infty \L_m$.
  \item[$3^\circ$] $\quad [\L_k,\L_l]\subseteq\bigoplus\limits_{m=k+l}^{k+l+
g}\L_m$.
\end{itemize}
\end{theorem}
\begin{proof}
The proof of the items $2^\circ$, $3^\circ$ is many times given in the works cited in the beginning of the section. The proof of the item $1^\circ$ is the only specific for $G_2$, and will be given now.

For a generic position of incoming and outgoing points the dimension $d_m$ of the space of all meromorphic $\g$-valued functions with the divisor $D_m$ is given by the Riemann--Roch theorem:
\[
  d_m=(\dim\g)(\deg\, D_m-g+1).
\]
The $\L_m$ is a subspace of the last space distinguished with the condition $\L_m\subset\L$ which is explicitly given by the relations \refE{orth1}, \refE{eigen}, \refE{resid}, \refE{L{-2}}, \refE{firstord}.

Thus,
\[
  \dim\L_m=d_m-\sharp(\text{relations}).
\]
Observe that
\[
  \deg D_m=-mN+m\sum_{i=1}^Ma_i+\sum_{i=1}^Mb_{m,i}+2K.
\]
By \refE{condd} we obtain
\[
  \deg D_m=-mN+mN+(N+g-1)+2K=N+g-1+2K.
\]
Hence $d_m=(\dim\g)(N+2K)$.

In order the statement $1^\circ$ of the theorem was true we must prove that the number of relations is equal to $2K\cdot\dim\g$, thus there must be $2\dim\g$ relations at every of $K$ $\ga$-points.

The matrix relation \refE{resid} prescribes the form of the residue. It reduces to $\dim\g$ scalar relations with $8$ free parameters $\b_{01}$, $\b_{02}$, $\b_1$, $\b_2$. In turn, the last two $3$-dimensional parameters are subjected to the $2$ linear relations \refE{orth1}. Thus \refE{resid} gives us $\dim\g-6$ effective relations.

In a similar way \refE{L{-2}} gives $\dim\g-1$ effective relations (a free parameter $\mu$ must be taken into account).

The \refE{eigen} gives $8$ linear relations with $2$ free parameters $\k_1$, $\k_2$,  i.e. $6$ effective relations.

Finally, we have one more relation \refE{firstord}.

We end up with $2\dim\g$ relations at every $\ga$-point, as required.
\end{proof}

The \refT{almgrad} defines an {\it almost graded structure} on
$\L$.

\section{Central extensions}
\label{S:centext}
 Let us recall from \cite{KNFa,KSlax,SSlax,Sh_DGr,Sch_Laxmulti} that a two-cocycle $\ga$ on $\L$ is called local if $\exists M\in\Z_+$ such that for any $m,n\in\Z$, $|m+n|>M$, and any $L\in\L_m$, $L'\in\L_n$ we have $\ga(L,L')=0$.  Our main goal in this section is the following theorem.
\begin{theorem}\label{T:central}
\begin{itemize}
  \item[]
  \item[$1^\circ $] For any $L,L'\in\L$ the 1-form $\tr(LdL'-\w[L,L'])$ is holomorphic except at the $P$- and $Q$-points where $\w$ is a $\g$-valued 1-form on $\Sigma$ defined below.
  \item[$2^\circ$] $\ga(L,L')=\sum\limits_{i=1}^N\res_{P_i}\tr(LdL'-\w[L,L'])$ gives a local cocycle on $\L$.
  \item[$3^\circ$] The almost-graded central extension of $\L$ given by the cocycle $\ga$ is unique up to equivalence and rescaling the central element.
\end{itemize}
\end{theorem}
The proof of the theorem is based on the following two lemmas.
\begin{lemma}\label{L:LdL'} For any $\ga$ the 1-form $\tr(LdL')$ has at most simple pole at $\ga$, and
\[
  \res_\ga\tr(LdL')=2(\k_1+\k_2)([L,L']).
\]
\end{lemma}
The lemma is proven in \refSS{LdL'pr}.

Let $\w$\label{defomega} be a $\g$-valued one-form possessing the following expansion at $\ga$-points:
\[
  \w=\w_{-1}\frac{dz}{z}+\w_0dz+\ldots
\]
where $\w_{-1}$, $\w_0$ have the form similar to \refE{resid}, \refE{L0}, respectively:
\[
\w_{-1}=\begin{pmatrix}
       0 & -\sqrt{2}\tilde\b_{02}\a_2^t & -\sqrt{2}\tilde\b_{01}\a_1^t \\
       \sqrt{2}\tilde\b_{01}\a_1 & \a_1\tilde\b_2^t-\tilde\b_1\a_2^t & \tilde\b_{02}[\a_2] \\
       \sqrt{2}\tilde\b_{02}\a_2 & \tilde\b_{01}[\a_1] & \a_2\tilde\b_1^t-\tilde\b_2\a_1^t \\
       \end{pmatrix},\quad
\w_0=\begin{pmatrix}
       0 & -\sqrt{2}w_2^t & -\sqrt{2}w_1^t \\
       \sqrt{2}w_1 & W & [w_2] \\
       \sqrt{2}w_2 & [w_1] & -W^t \\
       \end{pmatrix}
\]
where
\begin{equation}\label{E:orth2}
\a_1^t\tilde\b_2=1, \quad \a_2^t\tilde\b_1=1,
\end{equation}
\begin{equation}\label{E:eigen2}
\a_1^tw_2=0, \quad \a_2^tw_1=0, \quad W\a_1=\tilde\k_1\a_1,  \quad -W^t\a_2=\tilde\k_2\a_2,
\end{equation}
\begin{equation}\label{E:firstord2}
  \a_2^tW_1\a_1=0
\end{equation}
where $W_1$ is a $(2,2)$-block of $\w_1$.
\begin{lemma}\label{L:Lw} For any $\ga$ the 1-form $\tr(L\w)$ has at most simple pole at $\ga$, and
\[
  \res_\ga\tr(L\w)=2(\k_1+\k_2)(L).
\]
\end{lemma}
The lemma is proven in \refSS{Lwpr}.

\begin{proof}[Proof of the \refT{central}]
It follows from \refL{LdL'} and \refL{Lw} that $\tr(LdL')$ and $\tr(\w[L,L'])$ have at most simple poles at the $\ga$-points, and their residues are equal there. Hence  $\tr(LdL'-\w[L,L'])$ is holomorphic outside $P$- and $Q$-points. Let us denote the last 1-form by $\rho$: $\rho=\tr(LdL'-\w[L,L'])$.

Let $L\in\L_m$, $L'\in\L_{m'}$. Then by \refE{Dm}, \refE{homos}, and \refL{LdL'}
\[
  (LdL')\ge (m+m'-1)\sum\limits_{i=1}^NP_i-\sum_{j=1}^M(a_j(m+m')+b_{m,j}+b_{m',j}-1)\,Q_j+D_\ga
\]
where $D_\ga$ is a certain divisor supported at $\ga$-points.

Assume that $(\w)\ge \sum\limits_{i=1}^Nm_i^+P_i-\sum_{j=1}^Mm_j^-Q_j-\sum_{s=1}^K\ga_s$. Then
\[
  (\w[L,L'])\ge \sum\limits_{i=1}^N(m+m'+m_i^+)P_i-\sum_{j=1}^M(a_j(m+m')+b_{m,j}+b_{m',j}+m_j^-)\,Q_j+D'_\ga .
\]
where $D'_\ga$ is a certain divisor supported at $\ga$-points also.

In order $\rho=\tr(LdL'-\w[L,L'])$ had a nontrivial residue at least at one of the points $P_i$ it is necessary that
\[
  \min_{i=1,\ldots,N} \{ m+m'-1,m+m'+m_i^+\}\le -1,
\]
in other words,
\begin{equation}\label{E:upperb}
  m+m'\le -1-\min_{i=1,\ldots,N} \{ -1,m_i^+\}.
\end{equation}
On the other hand side, in order $\rho$ had a nontrivial residue at least at one of the points $Q_j$ it is necessary that
\[
  \max_{j=1,\ldots,M} \{a_j(m+m')+b_{m,j}+b_{m',j}-1 ,a_j(m+m')+b_{m,j}+b_{m',j}+m_j^-\}\ge 1.
\]
Since $b_{j,m}\le B, \forall j,m$ (see page \pageref{bjmbound}) the last inequality implies that
\[
  \max_{j=1,\ldots,M} \{a_j(m+m')+2B-1 ,a_j(m+m')+2B+m_j^-\}\ge 1,
\]
further on
\[
  \max_{j=1,\ldots,M} \{a_j(m+m')+2B-1 ,a_j(m+m')+2B+\max_{j=1,\ldots,M}m_j^-\}\ge 1,
\]
then
\[
  \max_{j=1,\ldots,M} \{a_j(m+m')\}\ge 1-\max\{  2B-1\, ,\, 2B+\! \max_{j=1,\ldots,M}m_j^- \},
\]
and finally
\begin{equation}\label{E:lowerb}
   m+m' \ge \min_{j=1,\ldots,M}\{ a_j^{-1}(1-\max\{  2B-1\, ,\, 2B+\! \max_{j=1,\ldots,M}m_j^- \})\}.
\end{equation}
By \refE{upperb} and \refE{lowerb} we conclude that $\ga(L,L')$ is a local cocycle.

The proof of uniqueness of the corresponding central extension (assertion $3^\circ$ of the theorem) has been given for an arbitrary simple Lie algebra $\g$ in \cite{SSlax} for the 2-point case ($N=1$) and in \cite{Sch_Laxmulti} for arbitrary sets of $P$- and $Q$-points.
\end{proof}
\section{Concluding remarks}\label{S:conclude}

1) The $G_2$ is the second example of a simple Lie algebra whose Lax operator algebra elements have double poles at the $\ga$-points. The first example is $\g=\spn(2n)$ \cite{KSlax}. It is also very similar to the case of $\spn(2n)$ that there is a first order relation \refE{firstord}. It is not clear what is the reason of different analytic behavior of elements of the Lax operator algebras corresponding to different types of simple Lie algebras. This behavior is not determined by the Dynkin scheme of $\g$. Indeed, $B_n$ and $C_n$ have the same Dynkin scheme but different pole orders at the $\ga$-points.

2) We did not touch any aspect of the relation between the Lax operator algebras over $G_2$ and integrable systems, though such relations certainly exist. For the classical Lie algebras, there is a duality between $\a$ and $\b$, $\ga$ and $\k$: they are dual canonical variables with respect to the Krichever--Phong symplectic form \cite{Klax,KrPhong,Sh0910_Hamilt,Sh_DGr}. For $\g=G_2$ this duality destroys: there are $2$ variables $\b_{01}$, $\b_{02}$ instead one having to be dual to $\a_0$, and $2$ variables $\k_1$, $\k_2$ having to be dual to $\ga$. We suppose that the only chance to define an analog of the Krichever--Phong symplectic structure is to do it on the subvariety $\k_1=\k_2$, $\b_{01}=\b_{02}$.  In this case the duality between $\a$ and $\b$, $\ga$ and $\k$ restores.
\appendix
\section{Proof of the closeness of $G_2$}\label{S:G2pr}
In this section we prove the closeness of $G_2$ with respect to the bracket. Let us take two arbitrary elements in $G_2$:
\[
\tilde A=\begin{pmatrix}
       0 & -\sqrt{2}a_2^t & -\sqrt{2}a_1^t \\
       \sqrt{2}a_1 & A & [a_2] \\
       \sqrt{2}a_2 & [a_1] & -A^t \\
       \end{pmatrix},\quad
       \tilde B=\begin{pmatrix}
       0 & -\sqrt{2}b_2^t & -\sqrt{2}b_1^t \\
       \sqrt{2}b_1 & B & [b_2] \\
       \sqrt{2}b_2 & [b_1] & -B^t \\
       \end{pmatrix}.
\]
Then
\[
 \begin{split}
 &[\tilde A,\tilde B]=\\
 &\left( \begin{array}{c|c|c}
       0 & \sqrt{2}(b_2^tA-a_2^tB-2(a_1\times b_1)^t)&\sqrt{2}(-b_1^tA^t+a_1^tB^t-2(a_2\times b_2)^t)\\
       \hline
       \sqrt{2}(Ab_1-Ba_1+2a_2\times b_2) &
       \begin{split}
            & AB-BA-3a_1b_2^t+3b_1a_2^t\\
           & -(b_1^ta_2)E+(b_2^ta_1)E
       \end{split} &  \begin{split}&-2a_1b_1^t+A[b_2]-[a_2]B^t \\
                                   &+2b_1a_1^t-B[a_2]+[b_2]A^t
                      \end{split}   \\ 
       \hline
       \sqrt{2}(-A^tb_2+B^ta_2+2a_1\times b_1)&\begin{split}&-2a_2b_2^t+[a_1]B-A^t[b_1]\\
                                                            &+2b_2a_2^t-[b_1]A+B^t[a_1]
                                                \end{split}
     & \begin{split}
            & A^tB^t-B^tA^t-3a_2b_1^t+3b_2a_1^t\\
           & +(b_1^ta_2)-(b_2^ta_1)E
       \end{split}   \\
       \end{array}\right).
 \end{split}
\]
To calculate the blocks $(2,2)$, $(3,3)$ we used the relation \refE{rels}$2^\circ$.

Let us check first that the blocks $(2,2)$, $(3,3)$ are traceless. For example, in the block $(2,2)$ we certainly have $\tr(AB-BA)=0$. Further on, $\tr(3a_1b_2^t)=3b_2^ta_1$, and $\tr(b_2^ta_1)E=3b_2^ta_1$ since this is  a $3\times 3$ scalar matrix. Hence $\tr(-3a_1b_2^t+(b_2^ta_1)E)=0$. Trace of the remainder of the block vanishes as well.

To complete the proof of closeness we only must check that the blocks $(3,2)$ and $(2,1)$, $(2,3)$ and $(1,2)$ are in correspondence respectively. For example, for the first pair we must prove that
\begin{equation}\label{E:check2}
  [Ab_1-Ba_1+2a_2\times b_2]=-2a_2b_2^t+[a_1]B-A^t[b_1]+2b_2a_2^t-[b_1]A+B^t[a_1].
\end{equation}
For this purpose we use the relations \refE{rels}$3^\circ$,\,$4^\circ$. By \refE{rels}$4^\circ$ $[2a_2\times b_2]=-2a_2b_2^t+2b_2a_2^t$. By \refE{rels}$^\circ$ we have $[Ab_1]=-A^t[b_1]-[b_1]A$, and $[-Ba_1]=[a_1]B+B^t[a_1]$. The \refE{check2} has been proven.
\section{Proof of the \refT{Liealg}}\label{S:close}

\subsection{Absence of order $-3$, $-4$ etc. terms}\label{SS:absence}
An order $-3$ term is equal to
\[
  L_{-3}^c=L_{-2}L'_{-1}+L_{-1}L'_{-2}-\leftrightarrow
\]
where $\leftrightarrow$ means the same expression  where the symbols with and without  $'$ are permuted. In our case $\leftrightarrow =L'_{-2}L_{-1}+L'_{-1}L_{-2}$.

\[
L_{-2}L'_{-1}=\begin{pmatrix}
       0 & 0 & 0\\
       0 & \a_1\a_2^t &  0   \\
       0 &  0  & -\a_2\a_1^t    \\
       \end{pmatrix}\begin{pmatrix}
       0 & -\sqrt{2}\b'_{02}\a_2^t & -\sqrt{2}\b'_{01}\a_1^t \\
       \sqrt{2}\b'_{01}\a_1 & \a_1{\b'_2}^t-\b'_1\a_2^t & \b'_{02}[\a_2] \\
       \sqrt{2}\b'_{02}\a_2 & \b'_{01}[\a_1] & \a_2{\b'_1}^t-\b'_2\a_1^t \\
       \end{pmatrix}.
\]
It vanishes\label{page2} by the orthogonality relations, and by $[\a_i]\a_i=\a_i\times\a_i=0$ $(i=1,2)$. Similarly $L_{-1}L'_{-2}=0$. Hence
\[
  L_{-3}^c=0.
\]
The same is true for $L_{-4}^c$ and lower terms.

\subsection{Order $-2$ term in the commutator}\label{SS:L{-2}}
Then the order $-2$ term of the expansion appears as the commutator $[L_{-1},L'_{-1}]$ where $L'$ has the same form, with all $\b$'s carrying an additional $'$. Let us calculate this commutator entry by entry, where every entry has to be first calculated for the product $L_{-1}L'_{-1}$, and then subjected to the alternation in symbols with and without $'$.

By the orthogonality conditions \refE{orth1} and relations $[\a_1]\a_1=\a_1\times\a_1=0$, $[\a_1]\a_2=\a_2\times\a_2=0$ the first column, and the first row in the product are equal to zero.

Further on, we have

$(2,2)=(-2\b_{01}\b'_{02}-\b_2^t\b'_1+\b_{02}\b'_{01})\a_1\a_2^t$ (here we used the relation $[\a_2][\a_1]=\a_1\a_2^t$).

$(3,3)=(-2\b_{02}\b'_{01}+\b_{01}\b'_{02}-\b_1^t\b'_2)\a_2\a_1^t$.

The two scalar coefficients in brackets after alternation will become opposite numbers, so that $(2,2)=\mu\a_1\a_2^t$, $(3,3)=-\mu\a_2\a_1^t$ where $\mu=2\b_{02}\b'_{01}-2\b_{01}\b'_{02}+\b_1{\b'_2}^t-\b_2^t\b'_1+\b_{02}\b'_{01}-\b_{01}\b'_{02}$.

As for the elements $(2,3)$ and $(3,2)$, they are equal to $0$. It is quite easy. Let us prove it for $(3,2)$ for example:

$(3,2)=-\sqrt{2}\b_{02}\a_2^t(\a_1{\b'_2}^t-\b'_1\a_2^t)-\sqrt{2}\b_{01}\a_1^t\b'_{01}[\a_1]$
The first summand vanishes by \refE{orth1}, and the second for the reason that \[\a_1^t[\a_1]=([\a_1]^t\a_1)^t=([-\a_1]\a_1)^t=-(\a_1\times\a_1)^t=0.
\]
Thus we obtained the following

\begin{equation}\label{E:L{-2}'}
   L_{-2}^c=[L_{-1},L'_{-1}]=\mu\begin{pmatrix}
       0 & 0 & 0\\
       0 & \a_1\a_2^t &  0   \\
       0 &  0  & -\a_2\a_1^t    \\
       \end{pmatrix}
\end{equation}
$L_{-2}$ is a coefficient in the order $-2$ term of the commutator, but the term of exactly the same form has to be added to the expansion of $L$ itself.
\subsection{Invariance of $L_{-1}$ with respect to the commutator}\label{SS:resid}
In the commutator
\[
 L_{-1}^c=L_0L'_{-1}+L_{-1}L'_0+L_1L'_{-2}+L_{-2}L'_1-\leftrightarrow .
\]
Let us calculate it term by term, entry by entry.
Take $L_0$ in the form \refE{L0}.
Then
\[
L_0L'_{-1}=\begin{pmatrix}
       0 & -\sqrt{2}a_2^t & -\sqrt{2}a_1^t\\
       \sqrt{2}a_1 & A &  [a_2]   \\
       \sqrt{2}a_2 &  [a_1]  & -A^t    \\
       \end{pmatrix}\begin{pmatrix}
           0   &   -\sqrt{2}\b'_{02}\a_2^t   &    -\sqrt{2}\b'_{01}\a_1^t \\
       \sqrt{2}\b'_{01}\a_1 & \a_1{\b'_2}^t-\b'_1\a_2^t & \b'_{02}[\a_2] \\
       \sqrt{2}\b'_{02}\a_2 & \b'_{01}[\a_1] & \a_2{\b'_1}^t-\b'_2\a_1^t \\
       \end{pmatrix}
\]
$
(1,1)=-2\b'_{01}a_2^t\a_1-2\b'_{02}a_1^t\a_2=0
$
by orthogonality relations.

\noindent
$
(2,1)=  \sqrt{2}\b'_{01}A\a_1 + \sqrt{2}\b'_{02}[a_2]\a_2 = \sqrt{2}\b'_{01}A\a_1+ \sqrt{2}\b'_{02}a_2\times\a_2 = \sqrt{2}\b'_{01}\k_1\a_1+ \sqrt{2}\b'_{02}a_2\times\a_2
$
by eigenvalue relations. We obtain
\[
  (2,1)=\sqrt{2}\left(\b'_{01}\k_1+\b'_{02}\l_1\right)\a_1.
\]
Similarly,
\[
  (3,1)=\sqrt{2}\left(\b'_{02}\k_2+\b'_{01}\l_2\right)\a_2.
\]
%

\noindent $(2,2)=-\sqrt{2}a_1\sqrt{2}\b'_{02}\a_2^t + A(\a_1{\b'_2}^t-\b'_1\a_2^t) +[a_2]\b'_{01}[\a_1]=
$\newline
$=-2\b'_{02}a_1\a_2^t + A(\a_1{\b'_2}^t-\b'_1\a_2^t) +\b'_{01}\a_1a_2^t=$\newline
$=\a_1(\k_1{\b'_2}^t+\b'_{01}a_2^t) -(2\b'_{02}a_1+A\b'_1)\a_2^t$.

We may don't calculate the remainder of entries since we know that the result belongs to $G_2$. So far we obtained the following:
\begin{equation}\label{E:L_0L'{-1}}
L_0L'_{-1}=\left(\begin{array}{c|c|c}
       0 & * &  *\\
       \hline
       \sqrt{2}\left(\b'_{01}\k_1+\b'_{02}\l_1\right)\a_1 &
       \begin{split}
            & \a_1(\k_1{\b'_2}^t+\b'_{01}a_2^t) \\
          - & (2\b'_{02}a_1+A\b'_1)\a_2^t
       \end{split} &  *   \\ 
       \hline
       \sqrt{2}\left(\b'_{02}\k_2+\b'_{01}\l_2\right)\a_2 &  *  &  *    \\
       \end{array}\right)
\end{equation}

Let us calculate now
\[
\begin{split}
&L_{-1}L'_0= \\
&=\begin{pmatrix}
           0   &   -\sqrt{2}\b_{02}\a_2^t   &    -\sqrt{2}\b_{01}\a_1^t \\
       \sqrt{2}\b_{01}\a_1 & \a_1{\b_2}^t-\b_1\a_2^t & \b_{02}[\a_2] \\
       \sqrt{2}\b_{02}\a_2 & \b_{01}[\a_1] & \a_2{\b_1}^t-\b_2\a_1^t \\
       \end{pmatrix}\begin{pmatrix}
       0 & -\sqrt{2}{a'_2}^t & -\sqrt{2}{a'_1}^t\\
       \sqrt{2}a'_1 & A' &  [a'_2]   \\
       \sqrt{2}a'_2 &  [a'_1]  & -{A'}^t    \\
       \end{pmatrix}\\
&=\left(\begin{array}{c|c|c}
       0 & * &  *\\
       \hline
       \begin{split}
         &\sqrt{2}((\a_1{\b_2}^t-\b_1\a_2^t)a'_1\\
         &+\b_{02}[\a_2]a'_2)
         \end{split} &
       \begin{split}
          -&2\b_{01}\a_1{a'_2}^t+(\a_1{\b_2}^t-\b_1\a_2^t)A'\\
          +& \b_{02}[\a_2][a'_1]
       \end{split} &  *   \\ 
       \hline
       \begin{split}
         &\sqrt{2}(\b_{01}[\a_1]a'_1\\
         &+(\a_2{\b_1}^t-\b_2\a_1^t)a'_2)
       \end{split}&  *  &  *    \\
       \end{array}\right)
\end{split}
\]
By $\a_2^ta'_1=0$, $[\a_2]a'_2=\a_2\times a'_2$ we have $(2,1)=\sqrt{2}(\a_1{\b_2}^ta'_1+\b_{02}\a_2\times a'_2)
= \sqrt{2}(\a_1{\b_2}^ta'_1-\b_{02}\l'_1\a_1)=\sqrt{2}({\b_2}^ta'_1-\b_{02}\l'_1)\a_1$.
Similarly, $(3,1)=\sqrt{2}({\b_1}^ta'_2-\b_{01}\l'_2 )\a_2$. Finally, $(2,2)=\a_1(-2\b_{01}{a'_2}^t+{\b_2}^tA')-\b_1\a_2^tA'+\b_{02}a'_1\a_2^t= \a_1(-2\b_{01}{a'_2}^t+{\b_2}^tA')+(\b_1\k'_2+\b_{02}a'_1)\a_2^t$.

Thus, we have
\[
 L_0L'_{-1}+L_{-1}L'_0=
 \left(\begin{array}{c|c|c}
       0 & * &  *\\
       \hline
        \sqrt{2}\tilde\b_{01}\a_1 & \a_1\tilde\b_2^t-\tilde\b_1\a_2^t   &  *   \\ 
       \hline
        \sqrt{2}\tilde\b_{02}\a_2 &  *  &  *    \\
       \end{array}\right)
\]
with
\[
\begin{split}
    & \tilde\b_{01}=\b'_{01}\k_1+\b'_{02}\l_1
      +\b_2^ta'_1-\b_{02}\l'_1,  \\
    & \tilde\b_{02}=\b'_{02}\k_2+\b'_{01}\l_2
      +\b_1^ta'_2-\b_{01}\l'_2, \\
    & \tilde\b_1=(2\b'_{02}a_1+A\b'_1) -(\b_1\k'_2+\b_{02}a'_1), \\
    & \tilde\b_2=(\k_1{\b'_2}+\b'_{01}a_2)+(-2\b_{01}a'_2+A'^t{\b_2}).
\end{split}
\]
We know that after alternating in symbols with and without $'$ we will get an element from $G_2$. Hence it will be a matrix of the form \refE{resid}. The only thing we must check is the orthogonality \refE{orth1} for $\tilde\b_1$, $\tilde\b_2$. Let us do it for $\a_2$ and $\tilde\b_1$ for example. Since $\a_2$ is orthogonal to $a_1$, $\b_1$, $a'_1$, we only must calculate $\a_2^tA\b'_1$. But $\a_2^tA=\k_2\a_2^t$, hence $\a_2^tA\b'_1=\k_2\a_2^t\b'_1=0$.

It is not an end of the story with the residue because we still have calculate $L_1L'_{-2}+L_{-2}L'_1-\leftrightarrow$.

Let us take $L_1$ in the form
\begin{equation}\label{E:L1}
           L_1=\begin{pmatrix}
                0 & -\sqrt{2}b_2^t & -\sqrt{2}b_1^t\\
                \sqrt{2}b_1 & B &  [b_2]   \\
                \sqrt{2}b_2 &  [b_1]  & -B^t    \\
                \end{pmatrix}.
\end{equation}
Then
\begin{equation}\label{E:L_1L'{-2}}
  \begin{split}
   L_1L'_{-2}&=\begin{pmatrix}
                0 & -\sqrt{2}b_2^t & -\sqrt{2}b_1^t\\
                \sqrt{2}b_1 & B &  [b_2]   \\
                \sqrt{2}b_2 &  [b_1]  & -B^t    \\
                \end{pmatrix}\cdot\mu'
                \begin{pmatrix}
                0 & 0 & 0\\
                0 & \a_1\a_2^t &  0   \\
                0 &  0  & -\a_2\a_1^t    \\
                \end{pmatrix}\\
            &=\mu'\left(\begin{array}{c|c|c}
       0 & -\sqrt{2}b_2^t\a_1\a_2^t &  \sqrt{2}b_1^t\a_2\a_1^t\\
       \hline
       0 & B\a_1\a_2^t   &  -[b_2]\a_2\a_1^t   \\ 
       \hline
       0 &  [b_1]\a_1\a_2^t  & B^t\a_2\a_1^t    \\
       \end{array}\right),
   \end{split}
\end{equation}
\[
  \begin{split}
   L_{-2}L'_1&= \mu\begin{pmatrix}
                0 & 0 & 0\\
                0 & \a_1\a_2^t &  0   \\
                0 &  0  & -\a_2\a_1^t    \\
                \end{pmatrix}\begin{pmatrix}
                0 & -\sqrt{2}{b'_2}^t & -\sqrt{2}{b'_1}^t\\
                \sqrt{2}b'_1 & B' &  [b'_2]   \\
                \sqrt{2}b'_2 &  [b'_1]  & -{B'}^t    \\
                \end{pmatrix}\\
            &=\mu\left(\begin{array}{c|c|c}
       0 & 0 & 0\\
       \hline
       \sqrt{2}\a_1\a_2^tb'_1 & \a_1\a_2^tB'   & \a_1\a_2^t[b'_2]    \\ 
       \hline
       -\sqrt{2}\a_2\a_1^tb'_2 & -\a_2\a_1^t[b'_1]  & \a_2\a_1^t{B'}^t   \\
       \end{array}\right)
   \end{split}
\]
\[
  \begin{split}
 &L_1L'_{-2}+L_{-2}L'_1=\\
     &=\left(\begin{array}{c|c|c}
       0 & -\sqrt{2}\mu' b_2^t\a_1\a_2^t &  \sqrt{2}\mu' b_1^t\a_2\a_1^t\\
       \hline
      \sqrt{2}\mu\a_1\a_2^tb'_1 &  \mu' B\a_1\a_2^t+\mu\a_1\a_2^tB'&-\mu'[b_2]\a_2\a_1^t+\mu\a_1\a_2^t[b'_2]   \\ 
       \hline
       -\sqrt{2}\a_2\a_1^t\mu b'_2 &  \mu' [b_1]\a_1\a_2^t-\a_2\a_1^t\mu [b'_1]  & \mu' B^t\a_2\a_1^t+ \mu\a_2\a_1^t{B'}^t    \\
       \end{array}\right).
   \end{split}
\]
After alternation we obtain
\[
  \begin{split}
 &L_1L'_{-2}+L_{-2}L'_1-\leftrightarrow\ =\\
     &=\left(\begin{array}{c|c|c}
       0 & \sqrt{2}(\mu{b'_2}^t-\mu' b_2^t)\a_1\a_2^t &  \sqrt{2}(\mu' b_1^t-\mu{b'_1}^t)\a_2\a_1^t\\
       \hline
      \sqrt{2}\a_2^t(\mu b'_1-\mu' b_1)\a_1 &  \begin{split}
                   &\a_1\a_2^t(\mu B'-\mu' B)\\-&(\mu B'-\mu' B)\a_1\a_2^t
                   \end{split}& \begin{split} -\mu'[b_2]\a_2\a_1^t+&\mu\a_1\a_2^t[b'_2]\\
                                                                  -&\leftrightarrow
                                \end{split}
                   \\ 
       \hline
       \sqrt{2}\a_1^t(\mu' b_2-\mu b'_2)\a_2 &  \begin{split}
                                                  \mu' [b_1]\a_1\a_2^t-&\a_2\a_1^t\mu [b'_1]\\-&\leftrightarrow\end{split}
        & \begin{split}
         &\a_2\a_1^t(\mu{B'}^t-\mu' B^t)\\
         &-(\mu{B'}^t-\mu' B^t)\a_2\a_1^t
         \end{split}
       \end{array}\right).
   \end{split}
\]
The last matrix is of the form \refE{resid} with
\[
  \b_1=(\mu B'-\mu' B)\a_1, \quad \b_2=(\mu{B'}^t-\mu' B^t)\a_2 .
\]
It is instructive to check correspondence between the blocks $(1,2)$ and $(2,3)$, $(2,1)$ and $(3,2)$ by a direct calculation, but we know in advance that $L_1L'_{-2}+L_{-2}L'_1-\leftrightarrow\in G_2$.

It is only necessary to check the orthogonality relations \refE{orth1}.
For this purpose we need one more assumption about the first order term of the  expansion for any $L$:
\begin{equation}\label{E:firstord}
  \a_2^tB\a_1=0.
\end{equation}
Then we have
\[
  \a_2^t\b_1=\a_2^t(\mu B'-\mu' B)\a_1=\mu\a_2^tB'\a_1-\mu'\a_2^tB\a_1=0.
\]
The second orthogonality relation is proven similarly.

For the classical simple Lie algebras there is the only case when the expansion for $L$ has an order $-2$ term, and in this case the analog of the relation \refE{firstord} was also needed. This is the case of the symplectic algebra.

\subsection{Eigenvalue conditions}\label{SS:eigen}
In this section our aim is to check the eigenvalue conditions in the form \refE{eigen} for the commutator. We have
\begin{equation}\label{E:eigen2}
 L_0^c= (L_{-2}L'_2+L_2L'_{-2})+(L_{-1}L'_1+L_1L'_{-1})+L_0L'_0\ -\leftrightarrow .
\end{equation}
Let us check \refE{eigen} for every expression in brackets separately.

Take $L_2$ in the form
\begin{equation}\label{E:L2}
  L_2=\begin{pmatrix}
                0 & -\sqrt{2}c_2^t & -\sqrt{2}c_1^t\\
                \sqrt{2}c_1 & C &  [c_2]   \\
                \sqrt{2}c_2 &  [c_1]  & -C^t    \\
                \end{pmatrix}.
\end{equation}
Then
\[
\begin{split}
  L_{-2}L'_2&=\mu\begin{pmatrix}
       0 & 0 & 0\\
       0 & \a_1\a_2^t &  0   \\
       0 &  0  & -\a_2\a_1^t    \\
       \end{pmatrix}\begin{pmatrix}
                0 & -\sqrt{2}{c'_2}^t & -\sqrt{2}{c'_1}^t\\
                \sqrt{2}c'_1 & C' &  [c'_2]   \\
                \sqrt{2}c'_2 &  [c'_1]  & -{C'}^t    \\
                \end{pmatrix}\\
                &=\mu\begin{pmatrix}
       0                        & *            & *\\
       \sqrt{2}(\a_2^tc'_1)\a_1 & \a_1\a_2^tC' &  *   \\
      -\sqrt{2}(\a_1^tc'_2)\a_2 &  *           & \a_2\a_1^tC'^t    \\
       \end{pmatrix}.
\end{split}
\]
Here only those entries have been explicitly written  which are independent in the result. All others are replaced with stars.

Similarly
\[
\begin{split}
  L_2L'_{-2}&=\mu'\begin{pmatrix}
                0 & -\sqrt{2}{c_2}^t & -\sqrt{2}{c_1}^t\\
                \sqrt{2}c_1 & C &  [c_2]   \\
                \sqrt{2}c_2 &  [c_1]  & -{C}^t    \\
                \end{pmatrix}\begin{pmatrix}
       0 & 0 & 0\\
       0 & \a_1\a_2^t &  0   \\
       0 &  0  & -\a_2\a_1^t    \\
       \end{pmatrix}\\
                &=\mu'\begin{pmatrix}
       0    & *            & *\\
       0    & C\a_1\a_2^t  &  *   \\
       0    &  *           & C^t\a_2\a_1^t
       \end{pmatrix},
\end{split}
\]
hence
\[
L_{-2}L'_2+L_2L'_{-2}=\left(\begin{array}{c|c|c}
       0                        & *            & *\\
       \hline
       \sqrt{2}\mu(\a_2^tc'_1)\a_1 & \mu\a_1\a_2^tC'+\mu' C\a_1\a_2^t &  *   \\
       \hline
      -\sqrt{2}\mu(\a_1^tc'_2)\a_2 &  *  & \mu\a_2\a_1^tC'^t+\mu'C^t\a_2\a_1^t    \\
       \end{array}\right) .
\]
Obviously, $L_{-2}L'_2+L_2L'_{-2}-\leftrightarrow$ has the form \refE{L0} with $a_1$ proportional to $\a_1$, $a_2$ proportional to $\a_2$, and $A=\mu\a_1\a_2^tC'+\mu' C\a_1\a_2^t\ - \leftrightarrow$. Hence $\a_1^ta_2=\a_2^ta_1=0$, and $A\a_1=\mu\a_1(\a_2^tC'\a_1)+\mu' C\a_1\a_2^t\a_1\ - \leftrightarrow$ which is proportional to $\a_1$. The remainder of the relations \refE{eigen}, $-A^t\a_2=\k_2\a_2$ can be proved similarly.

Further on,
\[
\begin{split}
  &L_{-1}L'_1
  =\begin{pmatrix}
       0 & -\sqrt{2}\b_{02}\a_2^t & -\sqrt{2}\b_{01}\a_1^t \\
       \sqrt{2}\b_{01}\a_1 & \a_1\b_2^t-\b_1\a_2^t & \b_{02}[\a_2] \\
       \sqrt{2}\b_{02}\a_2 & \b_{01}[\a_1] & \a_2\b_1^t-\b_2\a_1^t \\
       \end{pmatrix}\begin{pmatrix}
                0 & -\sqrt{2}{b'_2}^t & -\sqrt{2}{b'_1}^t\\
                \sqrt{2}b'_1 & B' &  [b'_2]   \\
                \sqrt{2}b'_2 &  [b'_1]  & -{B'}^t    \\
                \end{pmatrix}\\
  &=\left(\begin{array}{c|c|c}
       -2\b_{01}\a_1^tb'_2-2\b_{02}\a_2^tb'_1 & *            & *\\
       \hline
       \sqrt{2}(\a_1\b_2^t-\b_1\a_2^t)b'_1+\sqrt{2}\b_{02}[\a_2]b'_2  &
       \begin{split}
        &-2\b_{01}\a_1{b'_2}^t+\a_1\b_2^tB'\\
        &-\b_1\a_2^tB'+\b_{02}[\a_2][b'_1]
       \end{split}    &  *   \\
       \hline
          \sqrt{2}\b_{01}[\a_1]b'_1+\sqrt{2}(\a_2\b_1^t-\b_2\a_1^t)b'_2 &  *  & \begin{split}&-2\b_{02}\a_2{b'_1}^t+\b_{01}[\a_1][b'_2] \\
                       &-(\a_2\b_1^t-\b_2\a_1^t){B'}^t
          \end{split}\\
       \end{array}\right)
\end{split}
\]
Denote the entries $(2,1)$, $(3,1)$ and $(2,2)$ of the last matrix by $a_1$, $a_2$ and $A$ respectively, and prove the relations \refE{eigen} for $a_1,a_2,A$. If we succeed to do it for both $L_{-1}L'_1$ and $L_1L'_{-1}$, then these relations are true for  $L_{-1}L'_1+L_1L'_{-1}+\leftrightarrow$ since they are linear. Indeed
\[
  \a_2^ta_1=\a_2^t(\sqrt{2}(\a_1\b_2^t-\b_1\a_2^t)b'_1+\sqrt{2}\b_{02}[\a_2]b'_2).
\]
We have $\a_2^t\a_1=\a_2^t\b_1=0$, hence $\a_2^t(\a_1\b_2^t-\b_1\a_2^t)=0$. For any $x,y,z\in\C^3$ denote by $x\wedge y\wedge z$ the determinant of the $3
\times 3$-matrix constituted by those three vectors as columns. Then $\a_2^t[\a_2]b'_2=\a_2^t(\a_2\times b'_2)=\a_2\wedge\a_2\wedge b'_2=0$. We have obtained $\a_2^ta_1=0$. Similarly $\a_1^ta_2=0$.

It is the next step to prove the eigenvalue condition for the entry $(2,2)$ of the matrix. Observe that $[\a_2][b'_1]=b'_1\a_2^t-\a_2^tb'_1E$ where $E$ is the unit matrix. Let the entry $(2,2)$ operate on $\a_1$. We shall obtain
\[
  \begin{split}
 (2,2)\a_1&=(-2\b_{01}\a_1{b'_2}^t+\a_1\b_2^tB'-\b_1\a_2^tB'+\b_{02}[\a_2][b'_1])\a_1\\
          &=\a_1(-2\b_{01}{b'_2}^t\a_1+\b_2^tB'\a_1-\a_2^tb'_1)-\b_1\a_2^tB'\a_1+b'_1\a_2^t\a_1.
  \end{split}
\]
In the last line, the expression in brackets is a scalar, and the remainder of the line is equal to zero (by \refE{orth1}, \refE{firstord}). It is quite easy to prove that $(2,2)^t\a_2=0$, i.e. is also a multiple of $\a_2$ where $(2,2)^t$ denotes the $3\times 3$-matrix transposed to $(2,2)$.

Let us consider now $L_1L'_{-1}$. We have
\[
\begin{split}
  &L_1L'_{-1}=
  \begin{pmatrix}
                0 & -\sqrt{2}{b'_2}^t & -\sqrt{2}b_1^t\\
                \sqrt{2}b_1 & B &  [b_2]   \\
                \sqrt{2}b_2&  [b_1]  & -B^t    \\
                \end{pmatrix}\begin{pmatrix}
       0 & -\sqrt{2}\b'_{02}\a_2^t & -\sqrt{2}\b'_{01}\a_1^t \\
       \sqrt{2}\b'_{01}\a_1 & \a_1{\b'_2}^t-\b'_1\a_2^t & \b'_{02}[\a_2] \\
       \sqrt{2}\b'_{02}\a_2 & \b'_{01}[\a_1] & \a_2{\b'_1}^t-\b'_2\a_1^t \\
       \end{pmatrix}\\
  &=\left(\begin{array}{c|c|c}
       -2\b'_{01}\a_1^tb_2'-2\b'_{02}\a_2^tb_1'& *            & *\\
       \hline
          \sqrt{2}\b'_{01}B\a_1+\sqrt{2}\b'_{02}[b_2]\a_2   &
       \begin{split}
        & -2b_1\b'_{02}\a_2^t+B\a_1{\b'_2}^t \\
        & -B\b'_1\a_2^t+[b_2]\b'_{01}[\a_1]
       \end{split}    &  *   \\
       \hline
        \sqrt{2}\b'_{01}[b_1]\a_1-\sqrt{2}\b'_{02}B^t\a_2   &  *  &
        \begin{split}
          & -2\b'_{01}b_2\a_1^t+\b'_{02}[b_1][\a_2]  \\
          & -B^t(\a_2{\b'_1}^t-\b'_2\a_1^t)
        \end{split}    \\
       \end{array}\right).
\end{split}
\]
First, let us check the orthogonality relations starting with the element $(2,1)$. Observe that $\a_2^tB\a_1=0$ by \refE{firstord}, and  $\a_2^t[b_2]\a_2=0$ since $[b_2]$ is a skew-symmetric matrix. Hence $\a_2^t(\sqrt{2}\b'_{01}B\a_1+\sqrt{2}\b'_{02}[b_2]\a_2)=0$. Similarly $\a_1^t(\sqrt{2}\b'_{01}[b_1]\a_1-\sqrt{2}\b'_{02}B^t\a_2)=0$.

Next check the eigenvalue condition for the element $(2,2)$. We must operate on $\a_1$ by it. Observe that $\a_2^t\a_1={\b'_2}^t\a_1=0$ by \refE{orth1}, $[\a_1]\a_1=\a_1\times\a_1=0$. Hence $\a_1$ is an eigenvector of $(2,2)$ with  the eigenvalue $0$.

We may pass to the last summand in \refE{eigen2} now. We have
\[
\begin{split}
 L_0L'_0&=\begin{pmatrix}
       0 & -\sqrt{2}a_2^t & -\sqrt{2}a_1^t \\
       \sqrt{2}a_1 & A & [a_2] \\
       \sqrt{2}a_2 & [a_1] & -A^t \\
       \end{pmatrix}\begin{pmatrix}
       0 & -\sqrt{2}{a'_2}^t & -\sqrt{2}a_1'^t \\
       \sqrt{2}a'_1 & A' & [a'_2] \\
       \sqrt{2}a'_2 & [a'_1] & -{A'}^t \\
       \end{pmatrix}\\
       &=\left(\begin{array}{c|c|c}
        *                  & *            & *\\
       \hline
        \sqrt{2}Aa'_1+\sqrt{2}[a_2]a'_2      & -2a_1{a'_2}^t+AA'+[a_2][a'_1]    &  *   \\
       \hline
        \sqrt{2}[a_1]a'_1-\sqrt{2}A^ta'_2   &  *  & -2a_2a_1'^t+[a_1][a_2']+A^tA'^t  \\
       \end{array}\right).
\end{split}
\]
To prove the orthogonality relation for the entry $(2,1)$ observe first that $\a_2^tAa'_1=(A^t\a_2)^ta'_1=-\k_2\a_2^ta'_1=0$, and $\a_2^t[a_2]a'_2=\a_2^t(a_2\times a'_2)=\a_2\wedge a_2\wedge a'_2$. The last determinant is equal to $0$ because all three vectors $\a_2,a_2,a'_2$ are orthogonal to $\a_1$, hence they are linearly dependent. We have obtained that $\a_2^t(\sqrt{2}Aa'_1+\sqrt{2}[a_2]a'_2 )=0$. Similarly $\a_1^t(\sqrt{2}[a_1]a'_1-\sqrt{2}A^ta'_2)=0$.

To prove the eigenvalue relation for the entry $(2,2)$ we must operate on $\a_1$ by it. Observe first that $a'_2\a_1=0$, and second that $[a_2][a'_1]=a'_1a_2^t$, hence $[a_2][a'_1]\a_1=0$ again. For the remainder of the expression we have $AA'\a_1=\k_1\k'_1\a_1$. Similarly, the eigenvalue relation holds for $\a_2$ with respect to the transposed entry $(2,2)$.

\subsection{First order term condition}\label{SS:firstord}
In this section we prove that the relation \refE{firstord} keeps invariant under commutation. To this end, we must calculate the entry $(2,2)$ in the matrix
\[
  L^c_1=(L_{-2}L'_3+L_3L'_{-2})+(L_{-1}L'_2+L_2L'_{-1})+(L_0L'_1+L_1L'_0)-\leftrightarrow .
\]
We will keep the notation \refE{L0} for $L_0$, \refE{L1} for $L_1$, and \refE{L2} for $L_2$. Let us take $L_3$ in the form
\begin{equation}\label{E:L3}
  L_3=\begin{pmatrix}
                0 & -\sqrt{2}d_2^t & -\sqrt{2}d_1^t\\
                \sqrt{2}d_1 & D &  [d_2]   \\
                \sqrt{2}d_2 &  [d_1]  & -D^t    \\
                \end{pmatrix}.
\end{equation}
Then the entry $(2,2)$ in $L^c_1$ writes as follows:
\[
\begin{split}
  (2,2)&=(\mu\a_1\a_2^tD'+\mu'D\a_1\a_2^t)\\
       &+(-2\b_{01}{\a_1}c'_2+(\a_1\b_2^t-\b_1\a_2^t)C'+\b_{02}[\a_2][c_1])\\
       &+(-2\b_{02}c_1\a_2^t+C((\a_1{\b'_2}^t-{\b'_1}^t\a_2^t)+\b_{01}[c_2][\a_1]))\\
       &-2a_1{b'_2}^t+AB'+[a_2][b'_1]-2b_1{a'_2}^t+BA'+[b_2][a'_1].
\end{split}
\]
Having been multiplied by $\a_2^t$ from the left, and by $\a_1$ from the right, all summands of the firs line vanish by $\a_2^t\a_1=0$. The same is true for all summands of the second line, except the last one, if we additionally take account of $\a_2^t\b_1=0$.
As for the last summand, we have $[\a_2][c_1]=c_1\a_2^t-\a_2^tc_1E$, hence $\a_2^t[\a_2][c_1]\a_1=\a_2^tc_1\a_2^t\a_1-(\a_2^tc_1)\a_2^t\a_1=0$. The corresponding expression for the line three vanishes for the similar reason. In the fourth line, $\a_2^t(a_1{b'_2}^t+b_1{a'_2}^t)\a_1=0$ by the first two relations \refE{eigen}. Further on, $\a_2^tAB'\a_1=-\k_2\a_2^tB'\a_1=0$. The first equality by \refE{eigen}, and the second by \refE{firstord}. By the same argument $\a_2^tBA'\a_1=0$. For the remaining two terms we use the transformations $[a_2][b'_1]=b'_1a_2^t-(a_2^tb'_1)E$, and $[b_2][a'_1]=a'_1b_2^t-(b_2^ta'_1)E$ which again reduce the question to the first two relations \refE{eigen}.

The closeness of the space of the $L$-operators with respect to the bracket is proven.
\section{Proof of Lemmas \ref{L:LdL'}, \ref{L:Lw} }\label{S:centproof}
\subsection{Proof of \refL{LdL'}}\label{SS:LdL'pr}
\[
 \res_\ga LdL'=2L_{-2}L_2'+L_{-1}L_1'-L_1L'_{-1}-2L_2L'_{-2}.
\]
We are interested in the trace of this expression. All products are already calculated in \refSS{eigen}, in course of proving the \refT{Liealg}. We have
\[
 \begin{split}
 \tr L_{-1}L_1'=&-2\b_{01}\a_1^tb'_2-2\b_{02}\a_2^tb'_1+\\
                &\tr(-2\b_{01}\a_1{b'_2}^t+\a_1\b_2^tB'-\b_1\a_2^tB'+\b_{02}[\a_2][b'_1])+\\
                &\tr(-2\b_{02}\a_2{b'_1}^t+\b_{01}[\a_1][b'_2]-(\a_2\b_1^t-\b_2\a_1^t){B'}^t)\\
               =&-2\b_{01}\a_1^tb'_2-2\b_{02}\a_2^tb'_1+\\
                &-2\b_{01}\a_1^t{b'_2}+\b_2^tB'\a_1-\a_2^tB'\b_1+\b_{02}\tr(b'_1\a_2^t-\a_2^tb'_1E)+\\
                &-2\b_{02}\a_2^tb'_1-\a_2^t{B'}\b_1+\b_2^t{B'}\a_1+\b_{01}\tr(b'_2\a_1^t-\a_1^tb'_2E)\\
               =&-6\b_{01}\a_1^tb'_2-6\b_{02}\a_2^tb'_1+2\b_2^tB'\a_1-2\a_2^tB'\b_1.
                \end{split}
\]
By the symmetry of trace we have
\[
  \tr (L_{-1}L_1'-L_1L_{-1}')=\tr L_{-1}L_1'-\leftrightarrow .
\]
Further on,
\[
  \begin{split}\tr(2L_{-2}L'_2-2L_2L'_{-2})&=2\tr(\mu\a_1\a_2^tC'-\mu'
               C\a_1\a_2^t+\mu\a_2\a_1^tC'^t-\mu'C^t\a_2\a_1^t)\\
               &=4\mu\a_2^tC'\a_1-4\mu'\a_2^tC\a_1.
  \end{split}
\]
On the other hand side, let us calculate $\k_{1,2}([L,L'])$.
Above, we have computed
\[
\begin{split}
 L_0L'_0&=\\
       &=\left(\begin{array}{c|c|c}
        *                  & *            & *\\
       \hline
        \sqrt{2}Aa'_1+\sqrt{2}[a_2]a'_2      & -2a_1{a'_2}^t+AA'+[a_2][a'_1]    &  *   \\
       \hline
        \sqrt{2}[a_1]a'_1-\sqrt{2}A^ta'_2   &  *  & -2a_2a_1'^t+[a_1][a_2']+A^tA'^t  \\
       \end{array}\right).
\end{split}
\]
It's contribution $\k_1(L_0L'_0)$ is defined by the relation
\[
  (-2a_1{a'_2}^t+AA'+[a_2][a'_1])\a_1=\k_1(L_0L'_0)\a_1.
\]
We have
\[
  (-2a_1{a'_2}^t+AA'+[a_2][a'_1])\a_1=\k_1\k'_1\a_1+ (a'_1a_2^t-a_2^ta'_1E)\a_1.
\]
Hence
\[
  \k_1(L_0L'_0)=\k_1\k'_1-a_2^ta'_1.
\]
Similarly
\[
 \k_1(L'_0L_0)=\k'_1\k_1-a_2'^ta_1,
\]
hence
\[
  \k_1(L_0L'_0-L'_0L_0)=a_1^ta'_2-a_2^ta'_1.
\]
Further on,
\[
  (-2a_2a_1'^t+[a_1][a_2']+A^tA'^t)\a_2=(a_2'a_1^t-a_1^ta_2'E)\a_2+A^tA'^t\a_2,
\]
hence
\[
  \k_2(L_0L'_0)=-a_1^ta_2'+\k_2\k_2',\quad \k_2(L'_0L_0)=-a_2^ta_1'+\k'_2\k_2,
\]
and
\[
  \k_2(L_0L'_0-L'_0L_0)=a_2^ta_1'-a_1^ta_2'.
\]
We see that
\[
  \k_1(L_0L'_0-L'_0L_0)+\k_2(L_0L'_0-L'_0L_0)=0.
\]
Let us pass to the contribution of
\[
\begin{split}
  &L_{-1}L'_1
  =\\
  &=\left(\begin{array}{c|c|c}
       -2\b_{01}\a_1^tb'_2-2\b_{02}\a_2^tb'_1 & *            & *\\
       \hline
       \sqrt{2}(\a_1\b_2^t-\b_1\a_2^t)b'_1+\sqrt{2}\b_{02}[\a_2]b'_2  &
       \begin{split}
        &-2\b_{01}\a_1{b'_2}^t+\a_1\b_2^tB'\\
        &-\b_1\a_2^tB'+\b_{02}[\a_2][b'_1]
       \end{split}    &  *   \\
       \hline
          \sqrt{2}\b_{01}[\a_1]b'_1+\sqrt{2}(\a_2\b_1^t-\b_2\a_1^t)b'_2 &  *  & \begin{split}&-2\b_{02}\a_2{b'_1}^t+\b_{01}[\a_1][b'_2] \\
                       &-(\a_2\b_1^t-\b_2\a_1^t){B'}^t
          \end{split}\\
       \end{array}\right)
\end{split}
\]
By $\a_2^t\a_1=0$, $\k_1([\a_2][b_1'])=-\a_2^tb_1'$, $\k_2([\a_1][b_2'])=-\a_1^tb_2'$, and the relation \refE{firstord} we see that
\[
\begin{split}
  \k_1(L_{-1}L'_1)&=-2\b_{01}\a_1^tb_2'-\b_{02}\a_2^tb_1'+\b_2^tB'\a_1,\\ \k_2(L_{-1}L'_1)&=-\b_{01}\a_1^tb_2'-2\b_{02}\a_2^tb_1'-\a_2^tB'\b_1.
\end{split}
\]
It is easy to see that $\k_1(L_1L'_{-1})=\k_2(L_1L'_{-1})=0$. Hence
\[
  (\k_1+\k_2)(L_{-1}L'_1)=-3\b_{01}\a_1^tb_2'-3\b_{02}\a_2^tb_1'+\b_2^tB'\a_1-\a_2^tB'\b_1.
\]
The double of this expression coincides with the above expression for $\tr L_{-1}L_1'$.

Further on, we had
\[
L_{-2}L'_2+L_2L'_{-2}=\left(\begin{array}{c|c|c}
       0                        & *            & *\\
       \hline
       \sqrt{2}\mu(\a_2^tc'_1)\a_1 & \mu\a_1\a_2^tC'+\mu' C\a_1\a_2^t &  *   \\
       \hline
      -\sqrt{2}\mu(\a_1^tc'_2)\a_2 &  *  & \mu\a_2\a_1^tC'^t+\mu'C^t\a_2\a_1^t    \\
       \end{array}\right) .
\]
Hence
\[
  \k_1(L_{-2}L'_2+L_2L'_{-2})=\mu\a_2^tC'\a_1,\quad \k_2(L_{-2}L'_2+L_2L'_{-2})=
  \mu\a_1^tC'^t\a_2.
\]
Observe that the last two expressions are equal. Hence
\[
  (\k_1+\k_2)(L_{-2}L'_2+L_2L'_{-2})=2\mu\a_2^tC'\a_1,
\]
and
\[
  2(\k_1+\k_2)(L_{-2}L'_2+L_2L'_{-2}-\leftrightarrow)=4\mu\a_2^tC'\a_1-4\mu'\a_2^tC\a_1
\]
which coincides with the above expression for $\tr(2L_{-2}L'_2-2L_2L'_{-2})$.

We conclude that
\[
 \tr\res LdL'=2(\k_1+\k_2)([L,L']).
\]

Let us check now that $\tr(LdL')$ has at most simple poles at the $\ga$-points.

We have $(LdL')_{-5}=-2L_{-2}L'_{-2}$. This matrix is equal to $0$ by the relation $\a_1^t\a_2=0$.

The matrix $(LdL')_{-4}=-L_{-2}L'_{-1}-2L_{-1}L'_{-2}$ also vanishes as it is noticed at the page \pageref{page2}.

\[
  (LdL')_{-3}=-L_{-1}L'_{-1}-2L_0L'_{-2}.
\]
\[
\begin{split}
  L_{-1}L'_{-1}&=\begin{pmatrix}
       0 & -\sqrt{2}\b_{02}\a_2^t & -\sqrt{2}\b_{01}\a_1^t \\
       \sqrt{2}\b_{01}\a_1 & \a_1\b_2^t-\b_1\a_2^t & \b_{02}[\a_2] \\
       \sqrt{2}\b_{02}\a_2 & \b_{01}[\a_1] & \a_2\b_1^t-\b_2\a_1^t \\
       \end{pmatrix}\begin{pmatrix}
       0 & -\sqrt{2}\b'_{02}\a_2^t & -\sqrt{2}\b'_{01}\a_1^t \\
       \sqrt{2}\b'_{01}\a_1 & \a_1\b_2'^t-\b_1'\a_2^t & \b'_{02}[\a_2] \\
       \sqrt{2}\b'_{02}\a_2 & \b'_{01}[\a_1] & \a_2\b_1'^t-\b_2'\a_1^t \\
       \end{pmatrix}\\
       &=\left(\begin{array}{c|c|c}
       0        & *            & *\\
       \hline
       *        & (-2\b_{01}\b'_{02}-\b_2^t\b_1'+\b_{02}\b'_{01})\a_1\a_2^t      &  *   \\
       \hline
          *     &  *           &(-2\b_{02}\b'_{01}-\b_1^t\b_2'+\b_{01}\b'_{02})\a_2\a_1^t
       \end{array}\right).
\end{split}
\]
We have
\[
  \tr(L_{-1}L'_{-1})=0
\]
by $\a_1^t\a_2=0$.

Further on,
\[
  L_0L'_{-2}=\mu'\begin{pmatrix}
       0 & -\sqrt{2}a_2^t & -\sqrt{2}a_1^t \\
       \sqrt{2}a_1 & A & [a_2] \\
       \sqrt{2}a_2 & [a_1] & -A^t \\
       \end{pmatrix}\begin{pmatrix}
       0 & 0 & 0\\
       0 & \a_1\a_2^t &  0   \\
       0 &  0  & -\a_2\a_1^t    \\
       \end{pmatrix}=\begin{pmatrix}
       0 & 0 & 0\\
       0 & A\a_1\a_2^t  &  *   \\
       0 &  *  & A^t\a_2\a_1^t    \\
       \end{pmatrix}.
\]
By \refE{eigen} we have $A\a_1=\k_1\a_1$, $A^t\a_2=-\k_2\a_2$, hence $\tr(L_0L'_{-2})=0$\label{trL0L-2'}, and finally $\tr(LdL')_{-3}=0$.

\[
  (LdL')_{-2}=L_{-2}L_1'-L_0L'_{-1}-2L_1L'_{-2}.
\]
$L_{-2}L_1'$ has been calculated in the section \refSS{resid}. We have\label{L-2L1'}
\[
  \tr(L_{-2}L_1')=\a_2^tB'\a_1+\a_1^tB'^t\a_2=2\a_2^tB'\a_1=0
\]
(we made use of \refE{firstord} here).

The $L_0L'_{-1}$ is given by \refE{L_0L'{-1}}. Let us show that the trace of the block $(2,2)$ of that matrix is equal to zero. By \refE{L_0L'{-1}} $\tr(2,2)= (\k_1{\b'_2}^t+\b'_{01}a_2^t)\a_1 - \a_2^t(2\b'_{02}a_1+A\b'_1)$. By \refE{orth1}, \refE{eigen} ${\b'_2}^t\a_1=a_2^t\a_1=\a_2^t a_1=0$. Further on, $\a_2^tA\b'_1=(A^t\a_2)^t\b'_1=-\k_2\a_2^t\b'_1=0$. The same way, the trace of the block $(3,3)$ of $L_0L'_{-1}$ is equal to $0$. Hence $\tr L_0L'_{-1}=0$.

The $L_1L'_{-2}$ is given by \refE{L_1L'{-2}}. We have $\tr L_1L'_{-2}=\mu'(\a_2^tB\a_1+ \a_1^tB^t\a_2)=2\mu' \a_2^tB\a_1=0$ (we used \refE{firstord} here). We have $\tr(LdL')_{-2}=0$ as the result, and conclude that the 1-form $\tr LdL'$ has at most simple poles at the $\ga$-points.
\subsection{Proof of \refL{Lw}}\label{SS:Lwpr}
1) $(L\w)_{-3}=L_{-2}\w_{-1}$. We have
\[
\begin{split}
L_{-2}\w_{-1}&=\begin{pmatrix}
       0 & 0 & 0\\
       0 & \a_1\a_2^t &  0   \\
       0 &  0  & -\a_2\a_1^t    \\
       \end{pmatrix}\begin{pmatrix}
       0 & -\sqrt{2}\tilde\b_{02}\a_2^t & -\sqrt{2}\tilde\b_{01}\a_1^t \\
       \sqrt{2}\tilde\b_{01}\a_1 & \a_1\tilde\b_2^t-\tilde\b_1\a_2^t & \tilde\b_{02}[\a_2] \\
       \sqrt{2}\tilde\b_{02}\a_2 & \tilde\b_{01}[\a_1] & \a_2\tilde\b_1^t-\tilde\b_2\a_1^t \\
       \end{pmatrix}\\
       &=\begin{pmatrix}
       0        & 0            & 0\\
       0        & -(\a_2^t\tilde\b_1)\a_1\a_2^t      &  0   \\
       0        &  0           & (\a_1^t\tilde\b_2)\a_2\a_1^t
       \end{pmatrix},
\end{split}
\]
$\tr L_{-2}\w_{-1}=0$ by $\a_1^t\a_2=0$.

2) $(L\w)_{-2}=L_{-2}\w_0+L_{-1}\w_{-1}$. At page \pageref{trL0L-2'} we have shown that $\tr L_0L_{-2}'=0$. Taking account of the symmetry of trace, observe that $L_{-2}\w_0$ is a matrix of the same type, hence $\tr L_{-2}\w_0=0$.

\[
   L_{-1}\w_{-1}=\begin{pmatrix}
       0 & -\sqrt{2}\b_{02}\a_2^t & -\sqrt{2}\b_{01}\a_1^t \\
       \sqrt{2}\b_{01}\a_1 & \a_1\b_2^t-\b_1\a_2^t & \b_{02}[\a_2] \\
       \sqrt{2}\b_{02}\a_2 & \b_{01}[\a_1] & \a_2\b_1^t-\b_2\a_1^t \\
       \end{pmatrix}\begin{pmatrix}
       0 & -\sqrt{2}\tilde\b_{02}\a_2^t & -\sqrt{2}\tilde\b_{01}\a_1^t \\
       \sqrt{2}\tilde\b_{01}\a_1 & \a_1\tilde\b_2^t-\tilde\b_1\a_2^t & \tilde\b_{02}[\a_2] \\
       \sqrt{2}\tilde\b_{02}\a_2 & \tilde\b_{01}[\a_1] & \a_2\tilde\b_1^t-\tilde\b_2\a_1^t \\
       \end{pmatrix}
\]
Consider the diagonal blocks of this matrix. The entry $(1,1)$ is equal to zero.

We have
\[
\begin{split}
  (2,2)&=(\a_1\b_2^t-\b_1\a_2^t)(\a_1\tilde\b_2^t-\tilde\b_1\a_2^t)+
          \b_{02}\tilde\b_{01}[\a_2][\a_1] \\
       &=-(\b_2^t\tilde\b_1)\a_1\a_2^t+(\a_2^t\tilde\b_1)\b_1\a_2^t+\b_{02}\tilde\b_{01}\a_1\a_2^t .
\end{split}
\]
Hence $\tr (2,2)=0$ by $\a_1^t\a_2=0$, $\a_2^t\b_1=0$. Similarly $\tr (3,3)=0$, hence
\[
    \tr L_{-1}\w_{-1}=0.
\]
We conclude that $L\w$ has at most simple pole at $\ga$.

3)
\[
  \res_\ga L\w=L_{-2}\w_1+L_{-1}\w_0+L_0\w_{-1}.
\]
The matrix $L_{-2}\w_1$ is of the same type as $L_{-2}L_1'$ which is shown to be traceless at page \pageref{L-2L1'}. Hence $\tr L_{-2}\w_1=0$.

We have
\[
  L_{-1}\w_0=\begin{pmatrix}
       0 & -\sqrt{2}\b_{02}\a_2^t & -\sqrt{2}\b_{01}\a_1^t \\
       \sqrt{2}\b_{01}\a_1 & \a_1\b_2^t-\b_1\a_2^t & \b_{02}[\a_2] \\
       \sqrt{2}\b_{02}\a_2 & \b_{01}[\a_1] & \a_2\b_1^t-\b_2\a_1^t \\
       \end{pmatrix}\begin{pmatrix}
       0 & -\sqrt{2}w_2^t & -\sqrt{2}w_1^t \\
       \sqrt{2}w_1 & W & [w_2] \\
       \sqrt{2}w_2 & [w_1] & -W^t \\
       \end{pmatrix}
\]
Observe that $\a_1^tw_2=\a_2^tw_1=0$. This also implies that $\tr[\a_1][w_2]=\tr[\a_2][w_1]=0$. Indeed, $[\a_1][w_2]=w_2\a_1^t$, hence $\tr[\a_1][w_2]=\a_1^tw_2=0$.
Hence
\[
\begin{split}
  \tr(L_{-1}\w_0)&=\tr(\a_1\b_2^t-\b_1\a_2^t)W-\tr(\a_2\b_1^t-\b_2\a_1^t)W^t\\
            &=\b_2^tW\a_1-\a_2^tW\b_1-\b_1^tW^t\a_2+\a_1^tW^t\b_2\\
            &=\k_1\b_2^t\a_1+\k_2\a_2^t\b_1+\k_2\b_1^t\a_2+\k_1\a_1^t\b_2=0
\end{split}
\]
(we rely on \refE{orth1} here).

Further on
\[
L_0\w_{-1}=\begin{pmatrix}
       0 & -\sqrt{2}a_2^t & -\sqrt{2}a_1^t \\
       \sqrt{2}a_1 & A & [a_2] \\
       \sqrt{2}a_2 & [a_1] & -A^t \\
       \end{pmatrix}\begin{pmatrix}
       0 & -\sqrt{2}\tilde\b_{02}\a_2^t & -\sqrt{2}\tilde\b_{01}\a_1^t \\
       \sqrt{2}\tilde\b_{01}\a_1 & \a_1\tilde\b_2^t-\tilde\b_1\a_2^t & \tilde\b_{02}[\a_2] \\
       \sqrt{2}\tilde\b_{02}\a_2 & \tilde\b_{01}[\a_1] & \a_2\tilde\b_1^t-\tilde\b_2\a_1^t \\
       \end{pmatrix}.
\]
Hence
\[
\begin{split}
  \tr(L_0\w_{-1})&=\tr A(\a_1\tilde\b_2^t-\tilde\b_1\a_2^t)-\tr A^t(\a_2\tilde\b_1^t-\tilde\b_2\a_1^t)\\
            &=\tilde\b_2^tA\a_1-\a_2^tA\tilde\b_1-\tilde\b_1^tA^t\a_2+\a_1^tA^t\tilde\b_2\\
            &=\k_1\tilde\b_2^t\a_1+\k_2\a_2^t\tilde\b_1+\k_2\tilde\b_1^t\a_2+\k_1\a_1^t\tilde\b_2.
\end{split}
\]
By \refE{orth2} we finally obtain
\[
  \tr(L_0\w_{-1})=2(\k_1+\k_2).
\]



\end{document}